\documentclass[11pt]{article}

\usepackage{amsmath, amsthm, amssymb}
\usepackage{bm}
\usepackage{graphicx}
\usepackage{calc}
\usepackage{mathtools}
\usepackage{enumerate}
\usepackage{hyperref}
\usepackage[margin=1in]{geometry}

\title{The Product of Gaussian Matrices is Close to Gaussian} 

\author{Yi Li\footnote{Supported in part by Singapore Ministry of Education (AcRF) Tier 2 grant MOE2018-T2-1-013}\\
Division of Mathematical Sciences\\
Nanyang Technological University\\
yili@ntu.edu.sg
\and
David P. Woodruff\footnote{Supported in part by Office of Naval Research (ONR) grant N00014-18-1-256 and a Simons Investigator Award}\\
Department of Computer Science\\
Carnegie Mellon University\\
dwoodruf@andrew.cmu.edu
}

\newcommand*{\eqdefU}{\ensuremath{\mathop{\overset{\text{ }d\text{ }}{=}}}}
\newcommand*{\eqdist}{\mathop{\overset{\text{ }d\text{ }}{\resizebox{\widthof{\eqdefU}}{\heightof{=}}{=}}}}

\newcommand{\R}{\mathbb{R}}

\DeclareMathOperator*{\E}{\mathbb{E}}
\DeclareMathOperator{\tr}{tr}
\DeclareMathOperator*{\Var}{Var}
\DeclareMathOperator{\cov}{cov}

\newcommand{\norm}[1]{\left\| #1 \right\|}
\newcommand{\cD}{\mathcal{D}}
\newcommand{\cF}{\mathcal{F}}
\newcommand{\cG}{\mathcal{G}}
\newcommand{\cO}{\mathcal{O}}
\newcommand{\KL}{\mathrm{KL}}
\newcommand{\Tau}{\mathrm{T}}

\theoremstyle{plain}
\newtheorem{lemma}{Lemma}
\newtheorem{theorem}{Theorem}
\newtheorem{proposition}{Proposition}
\newtheorem{corollary}{Corollary}

\theoremstyle{definition}

\theoremstyle{remark}

\allowdisplaybreaks

\begin{document}

\maketitle

\begin{abstract}
    We study the distribution of the {\it matrix product} $G_1 G_2 \cdots G_r$ of $r$ independent Gaussian matrices of
    various sizes, where $G_i$ is $d_{i-1} \times d_i$, and we denote $p = d_0$, $q = d_r$, and require $d_1 = d_{r-1}$. Here
    the entries in each $G_i$ are standard normal random variables with mean $0$ and variance $1$. Such products arise in the
    study of wireless communication, dynamical systems, and quantum transport, among other places. We show that, provided
    each $d_i$, $i = 1, \ldots, r$, satisfies $d_i \geq C p \cdot q$, where $C \geq C_0$ for a constant $C_0 > 0$ 
    depending on $r$, then the matrix
    product $G_1 G_2 \cdots G_r$ has variation distance at most $\delta$ to a $p \times q$ matrix $G$ of i.i.d.\ standard
    normal random variables with mean $0$ and variance $\prod_{i=1}^{r-1} d_i$. Here $\delta \rightarrow 0$ as $C \rightarrow \infty$. Moreover, we show a converse for constant $r$ that
    if $d_i < C' \max\{p,q\}^{1/2}\min\{p,q\}^{3/2}$ for some $i$,
    then this total variation distance is at least $\delta'$, for an absolute constant $\delta' > 0$ depending on $C'$ and $r$. This converse is best possible when $p=\Theta(q)$.
\end{abstract}

\section{Introduction}
Random matrices play a central role in many areas of theoretical, applied, and computational mathematics. One particular
application is dimensionality reduction, whereby one often chooses a rectangular random matrix $G \in \mathbb{R}^{m \times n}$,
$m \ll n$, and computes $G \cdot x$ for a fixed vector $x \in \mathbb{R}^n$. Indeed, this is the setting in compressed
sensing and sparse recovery \cite{donoho2006compressed}, randomized numerical linear algebra \cite{kannan2009spectral,mahoney2011randomized,woodruff2014sketching}, 
and sketching algorithms for data streams
\cite{muthukrishnan2005data}. 
Often $G$ is chosen to be a Gaussian matrix, and in particular, an $m \times n$ matrix with entries that are i.i.d.\ 
normal random variables with mean $0$ and variance $1$, denoted by $N(0,1)$. Indeed, in compressed sensing, such matrices can be
shown to satisfy the Restricted Isometry Property (RIP) \cite{candes2008restricted}, while in randomized numerical linear algebra, in certain applications such as support vector machines \cite{paul2014random} and non-negative matrix factorization \cite{kapralov2016fake}, their performance is shown to often outperform that of other sketching matrices. 

Our focus in this paper will be on understanding the {\it product} of two or more Gaussian matrices. Such products arise
naturally in different applications. For example, in the over-constrained ridge regression 
problem $\min_x \|Ax-b\|_2^2 + \lambda \|x\|_2^2$, 
the design matrix $A \in \mathbb{R}^{n \times d}$, $n \gg d$, 
is itself often assumed to be Gaussian (see, e.g., \cite{obozinski2011support}). 
In this case, the ``sketch-and-solve'' algorithmic framework for regression \cite{sarlos2006improved}
would compute $G \cdot A$ and $G \cdot b$ for an $m \times n$ Gaussian matrix $G$ with $m \approx sd_{\lambda}$, where 
$sd_{\lambda}$ is the so-called statistical dimension \cite{alaoui2014fast}, and solve for the $x$
which minimizes $\|G \cdot A x - G \cdot b\|_2^2 + \lambda \|x\|_2^2$. 
While computing $G \cdot A$ is slower than computing the corresponding matrix product for other kinds of sketching matrices
$G$, it often has application-specific  \cite{paul2014random,kapralov2016fake} as well as statistical benefits \cite{raskutti2016statistical}. Notice
that $G \cdot A$ is the product of two independent Gaussian matrices, and in particular, $G$ has a small number of rows while
$A$ has a small number of columns -- this is precisely the rectangular case we will study below. Other applications in
randomized numerical linear algebra where the product of two Gaussian matrices arises is when one computes the product of a Gaussian
sketching matrix and Gaussian noise in a spiked identity covariance model \cite{yang2020reduce}. 

The product of two or more Gaussian matrices also arises in diverse fields such as multiple-input multiple-output (MIMO)
wireless communication channels \cite{muller2002asymptotic}. Indeed, 
similar to the above regression problem in which one wants to reconstruct an underlying vector $x$, 
in such settings one observes the vector $y = G_1 \cdots G_r \cdot x + \eta,$ where $x$ is the
transmitted signal and $\eta$ is background noise. This setting corresponds to the situation in which there are $r$ scattering
environments separated by major obstacles, and the dimensions of the $G_i$ correspond to the number of ``keyholes'' \cite{muller2002asymptotic}. 
To determine the mutual information of this channel, one needs to understand the singular values of the matrix $G_1 \cdots G_r$. 
If one can show the distribution of this product is close to that of a Gaussian distribution in total variation distance, then
one can use the wide range of results known for the spectrum of a single Gaussian matrix (see, e.g., 
\cite{vershynin2011introduction}). 
Other applications of products of Gaussian matrices include disordered spin chains \cite{crisanti2012products,bellman1954limit,ishitani1977central}, stability of large complex
dynamical systems \cite{may1972will,majumdar2014top}, symplectic maps and Hamiltonian mechanics \cite{crisanti2012products,benettin1984power,paladin1986scaling}, quantum transport in disordered wires \cite{mello1988macroscopic,iida1990statistical}, and
quantum chromodynamics \cite{osborn2004universal}; we refer the reader to \cite{ipsen2015products,akemann2015recent} for an overview. 

The main question we ask in this work is: 
\begin{center}
{\it What is the distribution of the product $G_1  G_2 \cdots G_r$ of $r$ independent Gaussian matrices of
    various sizes, where $G_i$ is $d_{i-1} \times d_i$?} 
\end{center}
Our main interest in the question above will be when $G_1$ has a small number $p = d_0$ of rows, and $G_r$ has a small number $q = d_{r}$ of columns. 
Despite the large body of work on random matrix theory (see, e.g., \cite{tao2012topics} 
for a survey), 
we are not aware of any work which attempts to bound the total variation distance of the entire distribution of 
$G_1  G_2 \cdots G_r$ to a Gaussian distribution itself.

\subsection{Our Results}
Formally, we consider the problem of distinguishing the product of normalized Gaussian matrices
\[
A_r = \left(\frac{1}{\sqrt{d_1}}G_1\right)\left(\frac{1}{\sqrt{d_2}}G_2\right)\cdots \left(\frac{1}{\sqrt{d_{r-1}}}G_{r-1}\right) \left(\frac{1}{\sqrt{d_1}}G_r\right)
\]
from a single normalized Gaussian matrix
\[
A_1 = \frac{1}{\sqrt{d_1}} G_1.
\]
We show that, when $r$ is a constant, with constant probability we cannot distinguish the distributions of these two random
matrices when $d_i\gg p,q$ for all $i$; and, conversely, with constant probability, we can distinguish these two distributions when the $d_i$ are not large enough. 
\begin{theorem}[Main theorem]
Suppose that $d_i\geq \max\{p,q\}$ for all $i$ and $d_{r-1} = d_1$.
\begin{enumerate}[(a)]
    \item It holds that \[
    d_{TV}(A_r, A_1) \leq C_1\sum_{i=1}^{r-1} \sqrt{\frac{pq}{d_i}},
    \]
    where $d_{TV}(A_r,A_1)$ denotes the total variation distance between $A_r$ and $A_1$, and $C_1 > 0$ is an absolute constant.
    \item If $p,q,d_1,\dots,d_r$ further satisfy that
    \[
    \sum_{j=1}^{r-1} \frac{1}{d_j}\geq\frac{C_2^r}{\max\{p,q\}^{\frac12}\min\{p,q\}^{\frac32}},
    \]
    where $C_2 > 0$ is an absolute constant, then $d_{TV}(A_r, A_1) \geq 2/3$.
\end{enumerate}
\end{theorem}

Part (a) states that $d_{TV}(A_r,A_1) < 2/3$ when $d_i\geq C_1' p q$ for all $i$ for a constant $C_1'$ depending on $r$. The converse in (b) implies that $d_{TV}(A_r,A_1) \geq 2/3$ when $d_i\leq C_2' \max\{p,q\}^{1/2}\min\{p,q\}^{3/2}$ for some $i$ for a constant $C_2'$ depending on $r$. When $p=\Theta(q)$ and $r$ is a constant, we obtain a dichotomy (up to a constant factor) for the conditions on $p,q$ and $d_i$.

\subsection{Our Techniques} 
\subparagraph{Upper Bound.} We start by explaining our main insight as to why the distribution of 
a product $G_1 \cdot G_2$ of a $p \times d$ matrix $G_1$ of i.i.d.\ $N(0,1)$ random variables and a $d \times q$ matrix $G_2$ of i.i.d. $N(0,1)$ random variables has low variation distance to 
the distribution of a $p \times q$ matrix $A$ of i.i.d.\ $N(0,d)$ random variables. One could try to directly understand
the probability density function as was done in the case of Wishart matrices in  \cite{bubeck2015testing,racz2016smooth}, which
corresponds to the setting when $G_1 = G_2$. 
However, there are certain algebraic simplifications in the case of the Wishart distribution that seem much less tractable when manipulating the density function of the product of independent
Gaussians \cite{Burda_2010}. Another approach would be to try to use entropic methods as in \cite{bubeck2018entropic,brennan2020phase}. Such
arguments try to reveal entries of the product $G_1 \cdot G_2$ one-by-one, arguing that for most conditionings of previous
entries, the new entry still looks like an independent Gaussian. However, the entries are clearly not independent -- if
$(G_1 \cdot G_2)_{i,j}$ has large absolute value, then $(G_1 \cdot G_2)_{i, j'}$ is more likely to be large in absolute
value, as it could indicate that the $i$-th row of $G_1$ has large norm. One could try to first condition on the norms of
all rows of $G_1$ and columns of $G_2$, but additional issues arise when one looks at submatrices: if $(G_1 \cdot G_2)_{i,j}, (G_1 \cdot G_2)_{i, j'},$ and $(G_1 \cdot G_2)_{i', j}$ are all large, then it could mean the
$i$-th row of $G_1$ and the $i'$-th row of $G_1$ are correlated with each other, since they both are correlated with the $j$-th
column of $G_2$. Consequently, since $(G_1 \cdot G_2)_{i, j'}$ is large, it could make it more likely that
$(G_1 \cdot G_2)_{i', j'}$ has large absolute value. This makes the entropic method difficult to apply in this context.

Our upper bound instead leverages beautiful work of Jiang \cite{Jiang_2006} and Jiang and Ma \cite{jiang:late} which bounds
the total variation distance between the distribution of an $r \times \ell$ submatrix of a random 
$d \times d$ orthogonal matrix (orthonormal rows and columns) and an $r \times \ell$ matrix with i.i.d.\ $N(0, 1/d)$ entries. Their
work shows that if $r \cdot \ell / d \to 0$ as $d\to \infty$, then the total variation
distance between these two matrix ensembles goes to $0$. It is not immediately clear how to apply such results in our context. 
First of all, which submatrix should we be looking at? Note though, that if $V^T$ is a $p \times d$ uniformly 
random (Haar measure) 
matrix with orthonormal rows, and $E$ is a $d \times q$ uniformly random matrix with orthonormal columns, then by rotational invariance,
$V^T E$ is identically distributed to a $p \times q$ submatrix of a $d \times d$ random orthonormal matrix. Thus, 
setting $r = p$ and $\ell = q$ in the above results, they imply that $V^T E$ is close in variation distance to a $p \times q$
matrix $H$ with i.i.d. $N(0,1/d)$ entries. Given $G_1$ and $G_2$, one could then write them in their {\it singular value decomposition}, obtaining $G_1 = U \Sigma V^T$ and $G_2 = E \Tau F^T$. Then $V^T$ and $E$ are independent and well-known
to be uniformly random $p \times d$ and $d \times q$ orthonormal matrices, respectively. Thus $G_1 \cdot G_2$ is close in total 
variation distance to $U \Sigma H \Tau F^T$. However, this does not immediately help either, as it is not clear what the distribution of this
matrix is. Instead, the ``right'' way to utilize the results above is to (1) observe that $G_1 \cdot G_2 = U \Sigma V^T G_2$ is identically distributed as $U \Sigma X$, where $X$ is a matrix of i.i.d.\ normal random variables, given the rotational invariance of the Gaussian
distribution. Then (2) $X$ is itself close to a product $W^T Z$ where $W^T$ is a random $p \times d$ matrix with orthonormal
rows, and $Z$ is a random $d \times q$ matrix with orthonormal columns, by the above results. Thus, $G_1 \cdot G_2$ is close
to $U \Sigma W^T Z$. Then (3) $U \Sigma W^T$ has the same distribution as $G_1$, so $U \Sigma W^T Z$ is close to $G_1' Z$, where
$G_1'$ and $G_1$ are identically distributed, and $G_1'$ is independent of $Z$. Finally, (4) $G_1' Z$ is identically distributed as a matrix $A_1$ of standard normal random variables because $G_1'$ is Gaussian and $Z$ has orthonormal columns, by rotational invariance
of the Gaussian distribution. 

We hope that this provides a general method for arguments involving Gaussian matrices - in step (2)
we had the quantity $U \Sigma X$, where $X$ was a Gaussian matrix, and then viewed $X$ as a product of a short-fat
random orthonormal matrix $W^T$ and a tall-thin random orthonormal matrix $Z$. 
Our proof for the product of more than $2$ matrices recursively uses similar ideas, and bounds the growth in variation distance
as a function of the number $r$ of matrices involved in the product. 

\subparagraph{Lower Bound.}
For our lower bound for constant $r$, we show that the fourth power of the Schatten 4-norm of a matrix, namely, $\|X\|_{S_4}^4 = \tr((X^T X)^2)$, can be used to distinguish a product $A_r$ of $r$ Gaussian matrices and a single Gaussian 
matrix $A_1$. We use Chebyshev's inequality, for which we need to find the expectation and variance of $\tr((X^T X)^2)$ for $X=A_r$ and $X=A_1$. 

Let us consider the expectation first. An idea is to calculate the expectation recursively, that is, for a fixed matrix $M$ and a Gaussian random matrix $G$ we express $\E\tr(((MG)^T (MG))^2)$ in terms of $\E\tr((M^T M)^2)$. The real situation turns out to be slightly more complicated. Instead of expressing $\E\tr(((MG)^T (MG))^2)$ in terms of $\E\tr((M^T M)^2)$ directly, we decompose $\E\tr(((MG)^T (MG))^2)$ into the sum of expectations of a few functions in terms of $M$, say,
\[
\E\tr(((MG)^T (MG))^2) = \E f_1(M) + \E f_2(M) + \cdots + \E f_s(M)
\]
and build up the recurrence relations for $\E f_1(MG),\dots,\E f_s(MG)$ in terms of $\E f_1(M)$, $\E f_2(M)$, ..., $\E f_s(M)$. It turns out that the recurrence relations are all linear, i.e.,
\begin{equation}\label{eqn:decomposition}
\E f_i(MG) = \sum_{j=1}^s a_{ij} \E f_j(M),\quad i = 1,\dots, s,
\end{equation}
whence we can solve for $\E f_i(A_r)$ and obtaining the desired expectation $\E\tr((A_r^T A_r)^2)$.

Now we turn to variance. One could try to apply the same idea of finding recurrence relations for $\Var(Q) = \E(Q^2) - (\E Q)^2$ (where $Q = \tr(((MG)^T(MG))^2)$), but it quickly becomes intractable for the $\E(Q^2)$ term as it involves products of eight entries of $M$, which all need to be handled carefully as to avoid any loose bounds; note, the subtraction of $(\E Q)^2$ is critically needed to obtain a small upper bound on $\Var(Q)$ and thus loose bounds on $\E(Q^2)$ would not suffice. For a tractable calculation, we keep the product of entries of $M$ to $4$th order throughout, without involving any terms of $8$th order.
To do so, we invoke the law of total variance,
\begin{equation}\label{eqn:law_of_total_variance_intro}
\Var_{M,G}(\tr((MG)^T(MG))^2) = \E_M\left(\left.\Var_G(\tr((G^T M^T M G)^2)) \right\vert M \right) + \Var_M\left(\left. \E_G \tr((G^T M^T M G)^2)\right\vert M\right).
\end{equation}
For the first term on the right-hand side, we use Poincar\'e's inequality to upper bound it. Poincar\'e's inequality for the Gaussian measure states that for a differentiable function $f$ on $\R^m$,
\[
\Var_{g\sim N(0,I_m)}(f(g)) \leq C\E_{g\sim N(0,I_m)}\|\nabla f(g)\|_2^2.
\]
Here we can simply let $f(X) = \tr((MX)^T(MX))^2)$ and calculate $\E\|\nabla f(G)\|_2^2$. This is tractable since $\E \|\nabla f(G)\|_2^2$ involves the products of at most $4$ entries of $M$, and we can use the recursive idea for the expectation above to express
\[
\E \|\nabla f(G)\|_2^2 = \sum_{i} a_{ij} \E g_i(M)
\]
for a few functions $g_i$'s and establish a recurrence relation for each $g_i$.

The second term on the right-hand side of \eqref{eqn:law_of_total_variance_intro} can be dealt with by plugging in \eqref{eqn:decomposition}, and turns out to depend on a new quantity $\Var(\tr^2(M^T M))$. We again apply the recursive idea and the law of total variance to 
\[
\Var_{M,G}(\tr^2(G^T M^T M G)) = \E_M\left(\left.\Var_G(\tr^2((G^T M^T M G)) \right\vert M \right) + \Var_M\left(\left. \E_G \tr^2(G^T M^T M G)\right\vert M\right). 
\]
Again, the first term on the right-hand side can be handled by Poincar\'e's inequality and the second-term turns out to depend on $\Var(\tr((M^T M)^2))$, which is crucial. 
We have now obtained a double recurrence involving inequalities on $\Var(\tr((M^T M)^2))$ and $\Var(\tr^2((M^T M)^2))$, from which we can solve for an upper bound on $\Var(\tr(A_r^T A_r)^2)$. This upper bound, however, grows exponentially in $r$, which is impossible to improve due to our use of Poincar\'e's inequality.

\section{Preliminaries}

\subparagraph{Notation.} For a random variable $X$ and a probability distribution $\cD$, we use $X\sim \cD$ to denote that $X$ is subject to $\cD$. For two random variables $X$ and $Y$ defined on the same sample space, we write $X\eqdist Y$ if $X$ and $Y$ are identically distributed.

We use $\cG_{m,n}$ to denote the distribution of $m\times n$ Gaussian random matrices of i.i.d.\ entries $N(0,1)$ and $\cO_{m,n}$ to denote the uniform distribution (Haar) of an $m\times n$ random matrix with  orthonormal rows. For a distribution $\cD$ on a linear space and a scaling factor $\alpha\in \R$, we use $\alpha\cD$ to denote the distribution of $\alpha X$, where $X\sim \cD$.

For two probability measures $\mu$ and $\nu$ on the Borel algebra $\cF$ of $\R^m$, the total variation distance between $\mu$ and $\nu$ is defined as 
\[
d_{TV}(\mu,\nu) = \sup_{A\in \mathcal{F}} |\mu(A) - \nu(A)| = \frac{1}{2}\int_{\R^m} \left|\frac{d\mu}{d\nu}-1\right| d\nu.
\]
If $\nu$ is absolutely continuous with respect to $\mu$, one can define the Kullback-Leibler Divergence between $\mu$ and $\nu$ as
\[
D_{\KL}(\mu \| \nu) = \int_{\R^m} \frac{d\mu}{d\nu}\log_2 \frac{d\mu}{d\nu} d\nu.
\]
If $\nu$ is not absolutely continuous with respect to $\mu$, we define $D_{\KL}(\mu \| \nu) = \infty$.

When $\mu$ and $\nu$ correspond to two random variables $X$ and $Y$, respectively, we also write $d_{TV}(\mu,\nu)$ and $D_{\KL}(\mu\|\nu)$ as $d_{TV}(X,Y)$ and $D_{\KL}(X \| Y)$, respectively.

The following is the well-known relation between the Kullback-Leibler divergence and the total variation distance between two probability measures.
\begin{lemma}[Pinsker's Inequality~{\cite[Theorem 4.19]{concentration_book}}]\label{lem:pinsker}
$d_{TV}(\mu,\nu) \leq \sqrt{\frac{1}{2}D_{\KL}(\mu \| \nu)}$.
\end{lemma}

The following result, concerning the distance between the submatrix of a properly scaled Gaussian random matrix and a submatrix of a random orthogonal matrix, is due to Jiang and Ma~\cite{jiang:late}. 

\begin{lemma}[{\cite{jiang:late}}]\label{lem:jiang_KL}
Let $G\sim \cG_{d,d}$ and $Z\sim \cO_{d,d}$. Suppose that $p,q\leq d$ and $\hat G$ is the top-left $p\times q$ block of $G$ and $\hat Z$ the top-left $p\times q$ block of $Z$. Then
\begin{equation}\label{eqn:jiang_KL}
d_{\KL}\left( \left. \frac{1}{\sqrt{d}}\hat G \right\| \hat Z \right) \leq C\frac{pq}{d},
\end{equation}
where $C>0$ is an absolute constant.
\end{lemma}

The original paper \cite{jiang:late} does not state explicitly the bound in \eqref{eqn:jiang_KL} and only states that the Kullback-Leibler divergence tends to $0$ as $d\to\infty$. A careful examination of the proof of \cite[Theorem 1(i)]{jiang:late}, by keeping track of the order of the various $o(1)$ terms, reveals the quantitative bound \eqref{eqn:jiang_KL}. 



\subparagraph{Useful Inequalities.} We list two useful inequalities below.
\begin{lemma}[Poincar\'e's inequality for Gaussian measure {\cite[Theorem 3.20]{concentration_book}}]
Let $X\sim N(0,I_n)$ be the standard $n$-dimensional Gaussian distribution and $f:\R^n\to\R$ be any continuously differentiable function. Then 
\[
\Var(f(X)) \leq \E\left(\norm{\nabla f(X)}_2^2\right).
\]
\end{lemma}

\begin{lemma}[Trace inequality, {\cite{trace_inequalities}}]\label{lem:trace_inequality}
Let $A$ and $B$ be symmetric, positive semidefinite matrices and $k$ be a positive integer. Then
\[
\tr((AB)^k) \leq \min\left\{ \norm{A}_{op}^k \tr(B^k),  \norm{B}_{op}^k \tr(A^k) \right\}.
\]
\end{lemma}

\section{Upper Bound}

Let $r\geq 2$ be an integer. Suppose that $G_1,\dots,G_r$ are independent Gaussian random matrices, where $G_i \sim \cG_{d_{i-1},d_i}$ and $d_0 = p$, $d_r = q$ and $d_{r-1}=d_1$. Consider the product of normalized Gaussian matrices
\begin{gather*}
A_r = \left(\frac{1}{\sqrt{d_1}}G_1\right)\left(\frac{1}{\sqrt{d_2}}G_2\right)\cdots \left(\frac{1}{\sqrt{d_{r-1}}}G_{r-1}\right) \left(\frac{1}{\sqrt{d_1}}G_r\right) \\
\shortintertext{and a single normalized Gaussian random matrix}
A_1 = \frac{1}{\sqrt{d_1}} G_1'
\end{gather*}
where $G_1' \sim \cG_{p,q}$. In this section, we shall show that when $p,q\ll d_i$ for all $i$, we cannot distinguish $A_r$ from $A_1$ with constant probability.

For notational convenience, let $W_i = \frac{1}{\sqrt{d_i}}G_i$ for $i\leq r$ and $W_r = \frac{1}{\sqrt{d_1}}G_r$. Assume that $pq \leq \beta d_i$ for some constant $\beta$ for all $i$. Our question is to find the total variation distance between the matrix product $W_1 W_2\cdots W_r$ and the product $W_1 W_r$ of two matrices. 

\begin{lemma}
Let $p,q,d,d'$ be positive integers satisfying that $pq \leq \beta d$ and $pq\leq \beta d'$ for some constant $\beta < 1$.
Suppose that $A\in \R^{p\times d}$, $G\sim \frac{1}{\sqrt{d}}\cG_{d,d'}$, and $L\sim \cO_{d',d}$. Further suppose that $G$ and $L$ are independent. Let $Z\sim \cO_{q,d}$ be independent of $A$, $G$ and $L$. Then
\[
d_{TV}(AGL, AZ^T) \leq C\sqrt{\frac{pq}{d}},
\]
where $C > 0$ is an absolute constant.
\end{lemma}
\begin{proof}
Let $A = U\Sigma V^T$ be its singular value decomposition, where $V$ has dimension $d\times p$. Then 
\[
A G L = U\Sigma (V^T G L) \eqdist U\Sigma X,
\] 
where $X$ is a $p\times q$ random matrix of i.i.d.\ $N(0,1/d)$ entries. Suppose that $\tilde{Z}$ consists of the top $p$ rows of $Z^T$. Then 
\[
A Z^T = U\Sigma (V^T Z^T) \eqdist U\Sigma \tilde{Z}.
\]
Note that $X$ and $Z$ are independent of $U$ and $\Sigma$. It follows from Lemma~\ref{lem:jiang_KL} that
\[
d_{\KL}(AGL \| AZ^T) = d_{\KL}(U\Sigma X \| U\Sigma \tilde{Z}) = d_{\KL}(X \| \tilde{Z}) \leq C\frac{pq}{d},
\]
where $C > 0$ is an absolute constant. The result follows from Pinsker's inequality (Lemma~\ref{lem:pinsker}).
\end{proof}

The next theorem follows from the above.
\begin{theorem} It holds that
\[
d_{TV}(W_1\cdots W_{r} , W_1 W_{r})  \leq C\sum_{i=1}^r \sqrt{\frac{pq}{d_i}},
\]
where $C>0$ is an absolute constant.
\end{theorem}
\begin{proof}
Let $W_{r} = U\Sigma V^T$ and $X_{i}\sim \cO_{q,d_i}$, independent from each other and from the $W_i$'s. Applying the preceding lemma with $A = W_1\cdots W_{r-2}$, $G = W_{r-1}$ and $L = U$, we have
\[
d_{TV}(W_1\cdots W_{r-2}W_{r-1} W_{r} , W_1\cdots W_{r-2} X_{r-1}^T \Sigma V^T) \leq C\sqrt{\frac{pq}{d_{r-1}}},
\]
Next, applying the preceding lemma with $A = W_1\cdots W_{r-3}$, $G = W_{r-1}$ and $L = X_r$, we have
\[
d_{TV}(W_1\cdots W_{r-2} X_r^T \Sigma V^T , W_1\cdots W_{r-3} X_{r-2}^T \Sigma V^T) \leq C\sqrt{\frac{pq}{d_{r-2}}},
\]
Iterating this procedure, we have in the end that 
\[
d_{TV}(W_1 W_2 X_3^T \Sigma V^T , W_1 X_2^T \Sigma V^T) \leq C\sqrt{\frac{pq}{d_2}}.
\]
Since $U$, $\Sigma$ and $V$ are independent and $X_2 \eqdist U^T$, it holds that $X_2^T\Sigma V^T \eqdist W_{r}$. Therefore,
\[
d_{TV}(W_1\cdots W_{r} , W_1 W_{r})  \leq C\sum_{i=2}^{r-1} \sqrt{\frac{pq}{d_i}}. \qedhere
\]
\end{proof}

Repeating the same argument for $W_1 W_r$, we obtain the following corollary immediately.
\begin{corollary}\label{cor:upperbound} It holds that
\[
d_{TV}(A_r, A_1)  \leq C\sum_{i=1}^{r-1} \sqrt{\frac{pq}{d_{i}}},
\]
where $C>0$ is an absolute constant.
\end{corollary}
\section{Lower Bound}

Suppose that $r$ is a constant. We shall show that one can distinguish the product of $r$ Gaussian random matrices 
\[
A_r = \left(\frac{1}{\sqrt{d_1}}G_1\right)\left(\frac{1}{\sqrt{d_2}}G_2\right)\cdots \left(\frac{1}{\sqrt{d_{r-1}}}G_{r-1}\right) \left(\frac{1}{\sqrt{d_1}}G_r\right),
\]
from one Gaussian random matrix
\[
A_1 = \frac{1}{\sqrt{d_1}}G_1'
\]
when the intermediate dimensions $d_1,\dots,d_{r-1}$ are not large enough. Considering $h(X) = \tr((X^T X)^2)$, it suffices to show that one can distinguish $h(A_r)$ and $h(A_1)$ with a constant probability for constant $r$. By Chebyshev's inequality, it suffices to show that
\[
\max\left\{\sqrt{\Var(h(A_1))}, \sqrt{\Var(h(A_r))}\right\} \leq c(\E h(A_r) - \E h(A_1))
\]
for a small constant $c$. We calculate that:
\begin{lemma}
Suppose that $r$ is a constant, $d_i\geq \max\{p,q\}$ for all $i=1,\dots,r$. When $p,q,d_1,\dots,d_r\to\infty$, 
\[
\E h(A_r) = \frac{pq(p+q+1)}{d_r^2} + (1+o(1))\frac{pq(p-1)(q-1)}{d_r^2}\sum_{j=1}^{r-1} \frac{1}{d_j}.
\]
\end{lemma}

\begin{lemma}
Suppose that $r$ is a constant, $d_i\geq \max\{p,q\}$ for all $i=1,\dots,r$. There exists an absolute constant $C$ such that, when $p,q,d_1,\dots,d_r$ are sufficently large,
\[
\Var(h(A_r)) \leq \frac{C^r(p^3 q + p q^3)}{d_1^4}.
\]
\end{lemma}

We conclude with the following theorem, which can be seen as a tight converse to Corollary~\ref{cor:upperbound} up to a constant factor on the conditions for $p,q,d_1,\dots,d_r$.

\begin{theorem}
Suppose that $r$ is a constant and $d_i\geq \max\{p,q\}$ for all $i=1,\dots,r$. Further suppose that $d_1=d_r$. When $p,q,d_1,\dots,d_r$ are sufficiently large and satisfy that
\[
\sum_{j=1}^{r-1} \frac{1}{d_j}\geq\frac{C^r}{\max\{p,q\}^{\frac12}\min\{p,q\}^{\frac32}},
\]
where $C > 0$ is some absolute constant, with probability at least $2/3$, one can distinguish $A_r$ from $A_1$.
\end{theorem}

\subsection{Calculation of the Mean}\label{sec:mean}

Suppose that $A$ is a $p\times q$ random matrix, and is rotationally invariant under left- and right-multiplication by orthogonal matrices. We define
\begin{align*}
S_1(p,q) &= \E A_{11}^4 \quad \text{(diagonal)} \\
S_2(p,q) &= \E A_{21}^4 \quad \text{(off-diagonal)} \\
S_3(p,q) &= \E A_{i1}^2 A_{j1}^2 \quad (i\neq j)\quad \text{(same column)}\\
S_4(p,q) &= \E A_{1i}^2 A_{1j}^2 \quad (i\neq j)\quad \text{(same row)}\\
S_5(p,q) &= \E A_{1i}^2 A_{2j}^2 \quad (i\neq j) \\
S_6(p,q) &= \E A_{ik} A_{il} A_{jk} A_{jl}\quad (i\neq j, k\neq l) \quad \text{(rectangle)}
\end{align*}
Since $A$ is left- and right-invariant under rotations, these quantities are well-defined. 
Then
\begin{gather*}
\begin{aligned}
\E \tr((A^TA)^2) = \E\sum_{1\leq i,j\leq q} (A^TA)_{ij}^2 &= \sum_{i=1}^q \E(A^TA)_{ii}^2 + \sum_{1\leq i, j\leq q,  i\neq j} \E(A^TA)_{ij}^2\\
&= q\E(A^T A)_{11}^2 + q(q-1)\E(A^T A)_{12}^2
\end{aligned}\\
\shortintertext{and}
\begin{aligned}
\E(A^T A)_{11}^2 = \E \bigg(\sum_{i=1}^p A_{i1}^2 \bigg)^2 &= \sum_{i=1}^p \E A_{i1}^4 + \sum_{1\leq i,j\leq p, i\neq j} \E A_{i1}^2 A_{j1}^2 \\
&= \E A_{11}^4 + (p-1)\E A_{21}^4 + p(p-1) \E A_{11}^2 A_{21}^2 \\
&=: S_1(p,q) + (p-1)S_2(p,q) + p(p-1)S_3(p,q)
\end{aligned}\\
\begin{aligned}
\E(A^T A)_{12}^2 = \E \bigg(\sum_{i=1}^p A_{i1}A_{i2} \bigg)^2 
&= \sum_{i=1}^p \E A_{i1}^2 A_{i2}^2 + \sum_{1\leq i,j\leq p, i\neq j} \E A_{i1} A_{i2} A_{j1} A_{j2} \\
&= p S_4(p,q) + p(p-1) S_6(p,q).
\end{aligned}
\end{gather*}
When $S_1(p,q)=S_2(p,q)$, we have
\begin{align*}
\E \tr((A^TA)^2) &= q(p S_1(p,q) + p(p-1)S_3(p,q)) + q(q-1)(p S_4(p,q) + p(p-1) S_6(p,q))\\
&= p q S_1(p,q) \!+\! pq(p\!-\!1)S_3(p,q) \!+\! pq(q\!-\!1)S_4(p,q) \!+\! p(p\!-\!1)q(q\!-\!1)S_6(p,q).
\end{align*}
When $A = G$, we have 
\begin{gather*}
S_1(p,q)=S_2(p,q) = 3, \quad S_3(p,q)=S_4(p,q)=S_5(p,q) = 1,\quad S_6(p,q) = 0 \\
\shortintertext{and so}
\E\tr((A^TA)^2) = 3pq + pq(p-1) + pq(q-1) = pq(p+q+1).
\end{gather*}
Next, consider $A = BG$, where $B$ is a $p\times d$ random matrix and $G$ a $d\times q$ random matrix of i.i.d.~$N(0,1)$ entries. The following proposition is easy to verify, and its proof is postponed to Appendix~\ref{sec:proof_of_simple_prop}.
\begin{proposition}\label{prop:simple_prop}
It holds that $\E A_{21}^4 =  \E A_{11}^4$.
\end{proposition}
\begin{proof}
We have
\begin{gather*}
\begin{aligned}
\E A_{11}^4 = \E\bigg(\sum_{i} B_{1i} G_{i1}\bigg)^4 &= \sum_{i,j,k,l} \E B_{1i}B_{1j}B_{1k}B_{1l} \E G_{i1}G_{j1}G_{k1}G_{l1} \\
&= 3\sum_i \E B_{1i}^4 + 3\sum_{i\neq j} \E B_{1i}^2 B_{1j}^2 \\
\end{aligned}
\shortintertext{and}
\begin{aligned}
\E A_{21}^4 = \E\bigg(\sum_{i} B_{2i} G_{i1}\bigg)^4 &= \sum_{i,j,k,l} \E B_{2i}B_{2j}B_{2k}B_{2l} \E G_{i1}G_{j1}G_{k1}G_{l1} \\
&= 3\sum_i \E B_{2i}^4 + 3\sum_{i\neq j} \E B_{2i}^2 B_{2j}^2 = \E A_{11}^4. \qedhere
\end{aligned}
\end{gather*}
\end{proof}

Suppose that the associated functions of $B$ are named $T_1,T_2,T_3,T_4,T_6,T_5$. Then we can calculate that (detailed calculations can be found in Appendix~\ref{sec:omitted_mean})
\begin{align*}
S_1(p,q) &= 3d T_1(p,d) + 3d(d-1)T_4(p,d)\\
S_3(p,q) &= 3dT_3(p,d) + d(d-1)T_5(p,d) + 2d(d-1)T_6(p,d)\\
S_4(p,q) &= d T_1(p,d) + d(d-1)T_4(p,d)\\
S_5(p,q) &= d T_3(p,d) + d(d-1)T_5(p,d)\\
S_6(p,q) &= d T_3(p,d) + d(d-1)T_6(p,d)
\end{align*}
It is clear that $S_1,S_3,S_4,S_5,S_6$ depend only on $d$ (not on $p$ and $q$) if $T_1,T_3,T_4,T_5,T_6$ do so. Furthermore, if $T_1=3T_4$ then we have $S_1=3S_4$ and thus $S_4 = d(d+2)T_4$. If $T_3=2T_6+T_5$ then $S_3=d(d+2)T_3$ and $S_3 = 2S_6+S_5$. Hence, if $T_3=T_4$ then $S_3=S_4$. We can verify that all these conditions are satisfied with one Gaussian matrix and we can iterate it to obtain these quantities for the product of $r$ Gaussian matrices with intermediate dimensions $d_1,d_2,\dots,d_{r-1}$. We have that
\[
S_3 = S_4 = \prod_{i=1}^{r-1} d_i(d_i+2),\quad
S_1 = 3S_4,\quad
S_6 = \sum_{j=1}^{r-1} \left(\prod_{i=1}^{j-1} d_i(d_i+2)\right)d_j\left(\prod_{i=j+1}^{r-1} d_i(d_i-1)\right).
\]
Therefore, normalizing the $i$-th matrix by $1/\sqrt{d_i}$, that is,
\[
A = \left(\frac{1}{\sqrt{d_1}}G_1\right)\left(\frac{1}{\sqrt{d_2}}G_2\right)\cdots \left(\frac{1}{\sqrt{d_{r-1}}}G_{r-1}\right) \left(\frac{1}{\sqrt{d_1}}G_r\right),
\]
we have for constant $r$ that
\begin{equation}\label{eqn:mean_estimate}
\begin{aligned}
\E\tr((A^TA)^2) &= \frac{1}{d_1^2 d_2^2 \cdots d_{r-1}^2 d_1^2}\left(pq(p+q+1)S_3 + pq(p-1)(q-1)S_6\right) \\
&\approx \frac{pq(p+q+1)}{d_r^2} + \frac{pq(p-1)(q-1)}{d_r^2}\sum_{j=1}^{r-1}\frac{1}{d_j}.
\end{aligned}
\end{equation}

\subsection{Calculation of the Variance}\label{sec:variance}

Let $M\in \R^{p\times p}$ be a random symmetric matrix, and let $G\in \R^{p\times q}$ be a random matrix of i.i.d.\ $N(0,1)$ entries. We want to find the variance of $\tr((G^T MG)^2)$. The detailed calculations of some steps can be found in Appendix~\ref{sec:omitted_variance}.

Our starting point is the law of total variance, which states that 
\begin{equation}\label{eqn:law of total variance}
\Var(\tr((G^T MG)^2)) = \E_M \left(\left.\Var_G (\tr((G^T MG)^2)) \right\vert M \right) +  \Var_M\left(\left. \E_G \tr((G^T MG)^2)\right\vert M\right) 
\end{equation}

\subparagraph*{Step 1a.} We shall handle each term separately. Consider the first term, which we shall bound using the Poincar\'e inequality for Gaussian measures. Define $f(X) = \tr((X^T MX)^2)$, where $X\in \R^{p\times q}$. We shall calculate $\nabla f$. 
\[
f(X) = \norm{X^T M X}_F^2 = \sum_{1\leq i,j\leq q} (X^T MX)_{ij}^2 = \sum_{1\leq i,j\leq q}\bigg(\sum_{1\leq k,l\leq p} M_{k l} X_{k i} X_{l j}\bigg)^2.
\]
Then
\[
\frac{\partial f}{\partial X_{rs}} = \sum_{1\leq i,j\leq q} 2\bigg(\sum_{1 \leq u,v\leq p} M_{u v} X_{u i} X_{v j}\bigg)\bigg(\sum_{1\leq k,l\leq p}\frac{\partial}{\partial X_{rs}}  (M_{k l} X_{k i} X_{l j})\bigg).
\]
Note that
\[
\frac{\partial}{\partial X_{rs}}  (M_{k l} X_{k i} X_{l j}) = \begin{cases}
														M_{k l} X_{l j}, & (k,i)=(r,s)\text{ and }(l,j)\neq(r,s)\\
														M_{k l} X_{k i}, & (k,i)\neq (r,s)\text{ and }(l,j)=(r,s)\\
														2M_{r r} X_{r s}, & (k,i)=(r,s)\text{ and }(l,j)=(r,s)\\
														0, & \text{otherwise}.
													   \end{cases}
\]
we have that
\begin{align*}
\frac{\partial f}{\partial X_{rs}} &= 4\left(\sum_{1 \leq u,v\leq p} M_{u v} X_{u s} X_{v s}\right) M_{r r} X_{r s} + 2\sum_{(l,j)\neq (r,s)} \left(\sum_{1 \leq u,v\leq p} M_{u v} X_{u s} X_{v j}\right) M_{r l} X_{l j}\\
&\qquad + 2\sum_{(k,i)\neq (r,s)} \left(\sum_{1 \leq u,v\leq p} M_{u v} X_{u i} X_{v s}\right) M_{k r} X_{k i}\\
&=4\left[ \left(\sum_{1 \leq u,v\leq p} M_{u v} X_{u s} X_{v s}\right) M_{r r} X_{r s} + \sum_{(l,j)\neq (r,s)} \left(\sum_{u,v} M_{u v} X_{u s} X_{v j}\right) M_{r l} X_{l j} \right] \\
&= 4\sum_{l,j} \left(\sum_{u,v} M_{u v} X_{u s} X_{v j}\right) M_{r l} X_{l j}.
\end{align*}
Next we calculate $\E (\partial f/\partial X_{rs})^2$ when $X$ is i.i.d.\ Gaussian. 
\[
\left(\frac{1}{4}\frac{\partial f}{\partial X_{rs}}\right)^2 = \sum_{\substack{l,j\\ l',j'}} \sum_{\substack{u,v\\ u',v'}} M_{u v} M_{u' v'} M_{r l}  M_{r l'} \E X_{u s} X_{u' s} X_{v j} X_{l j} X_{v' j'}   X_{l' j'}
\]
We discuss different cases  of $j,j',s$.

When $j\neq j'\neq s$, it must hold that $u=u'$, $v=l$ and $v'=l'$ for a possible nonzero contribution, and the total contribution in this case is at most $q(q-1)B^{(1)}_{r,s}$, where
\[
B^{(1)}_{r,s} = \sum_{1\leq l,l'\leq p} \sum_u M_{ul} M_{ul'} M_{rl} M_{rl'} = \sum_u \langle M_{u,\cdot},M_{r,\cdot}\rangle^2.
\]

When $j = j'\neq s$, it must hold that $u = u'$ for a possible nonzero contribution, and the total contribution in this case is at most $(q-1)B^{(2)}_{r,s}$, where 
\begin{align*}
B^{(2)}_{r,s} &= \sum_{l,l'} \sum_{u,v,v'} M_{uv}M_{uv'}M_{rl} M_{rl'}\E X_{us}^2 X_{vj} X_{lj} X_{v'j} X_{l'j} \\
&= \norm{M}_F^2 \norm{M_{r,\cdot}}_2^2 + 2\sum_u \langle M_{u,\cdot},M_{r,\cdot}\rangle^2.
\end{align*}

When $j = s\neq j'$, it must hold that $v' = l'$ for possible nonzero contribution, and the total contribution in this case is at most $(q-1)B^{(3)}_{r,s}$, where
\begin{align*}
B^{(3)}_{r,s} &= \sum_{j'\neq s}\left[\sum_{l, l'} \sum_{u,v} M_{u v} M_{u'l'} M_{rl} M_{rl'}\E X_{us} X_{u's} X_{vs} X_{ls} X_{l'j'}^2\right]\\
&= \sum_{l, l'} (2\langle M_{l,\cdot}, M_{l',\cdot}\rangle + \tr(M) M_{l l'}) M_{rl} M_{rl'}.
\end{align*}
When $j = j' = s$, the nonzero contribution is
\[
B^{(4)}_{r,s} = \sum_{l, l'} \sum_{\substack{u,v\\ u',v'}} M_{u v} M_{u' v'} M_{r l}  M_{r l'} \E X_{u s} X_{u' s} X_{v s} X_{l s} X_{v' s} X_{l' s}.
\]
Since $u, u', v, v', l, l'$ needs to be paired, the only case which is not covered by $B^{(1)}_{rs}, B^{(3)}_{rs}$ and $B^{(3)}_{rs}$ is when $u=v$, $u'=v'$ and $l=l'$, in which case the contribution is at most
\[
	\sum_l \sum_{u,u'} M_{uu} M_{u'u'} M_{rl}^2 \E X_{us}^2 X_{u's}^2 X_{ls}^2 \lesssim \tr^2(M) \norm{M_{r,\cdot}}_2^2.
\]
Hence
\[
B^{(4)}_{r,s} \lesssim B^{(1)}_{r,s} + B^{(2)}_{r,s} + B^{(3)}_{r,s} + \tr^2(M) \norm{M_{r,\cdot}}_2^2.
\]
It follows that
\begin{gather*}
\sum_{r,s} B^{(1)}_{r,s} = q \sum_{u,r} \langle M_{u,\cdot},M_{r,\cdot}\rangle^2 = q\tr(M^4) \\
%
\sum_{r,s} B^{(2)}_{r,s} = q\sum_r \norm{M}_F^2 \norm{M_{r,\cdot}}_2^2 + 2q \sum_{u,r} \langle M_{u,\cdot},M_{r,\cdot}\rangle^2 = q \norm{M}_F^4 + 2q\tr(M^4) \\
%
\begin{aligned}
\sum_{r,s} B^{(3)}_{r,s} &= \sum_{r,s} \sum_{l, l'}(2\langle M_{l,\cdot}, M_{l',\cdot}\rangle + \tr(M) M_{l'l'}) M_{rl} M_{rl'}\\
&\leq 2q\tr(M^4) + q \tr(M) \norm{M}_F \sqrt{\tr(M^4)}
\end{aligned}
\end{gather*}
Note that $\tr(M^4) \leq \tr^2(M^2) = \norm{M}_F^4$. Hence
\begin{align*}
\frac{1}{16}\E \norm{\nabla f}_2^2 &\leq \sum_{r,s} ((q-1)(q-2)B_{rs}^{(1)} + (q-1)B_{rs}^{(2)} + (q-1)B_{rs}^{(3)} + B_{rs}^{(4)}) \\
&\lesssim \sum_{r,s} (q^2 B_{rs}^{(1)} + q B_{rs}^{(2)} + q B_{rs}^{(3)} + \tr^2(M) \norm{M_{r,\cdot}}_2^2)\\
&\lesssim q^3 \tr(M^4) + q^2 \norm{M}_F^4 + q^2 \tr(M) \norm{M}_F \sqrt{\tr(M^4)} + q\tr^2(M)\norm{M}_F^2.
\end{align*}
By the Gaussian Poincar\'e inequality,
\begin{equation}\label{eqn:Var_G term in U}
\begin{aligned}
&\quad\ \Var_G (\tr((G^T MG)^2) | M) \\
&\lesssim \E\norm{\nabla f}_2^2 \\
&\lesssim q^3 \tr(M^4) + q^2 \norm{M}_F^4 + q^2 \tr(M) \norm{M}_F \sqrt{\tr(M^4)} + q\tr^2(M)\norm{M}_F^2.
\end{aligned}
\end{equation}
For the terms on the right-hand side, we calculate that (using the trace inequality (Lemma~\ref{lem:trace_inequality}))
\begin{gather*}
\begin{aligned}
\E\tr((G^T MG)^4) = \E\tr((MGG^T)^4) \leq \E\norm{GG^T}_{op}^4 \tr(M^4) &= \E\norm{G}_{op}^8 \tr(M^4) \\
&\lesssim \max\{p,q\}^4\tr(M^4),
\end{aligned}
\\
\E\norm{G^T MG}_F^4 \leq \E\norm{G}_{op}^8 \norm{M}_F^4 \lesssim \max\{p,q\}^4  \norm{M}_F^4,
\\
\E \tr^2(G^T MG)\norm{G^T MG}_F^2\leq \E \norm{G}_{op}^8 \tr^2(M)\norm{M}_F^2 \lesssim \max\{p,q\}^4 \tr^2(M)\norm{M}_F^2
\end{gather*}
and
\begin{align*}
&\quad\ \E \tr(G^T MG) \norm{G^T MG}_F \sqrt{\tr((G^T MG)^4)}  \\
&\leq \E \norm{G}_{op}^2 \tr(G) \cdot \norm{G}_{op}^2 \norm{M}_F^2 \cdot \sqrt{\norm{G}_{op}^8 \tr(M^4)}\\
&= \E\norm{G}_{op}^8 \tr(M) \norm{M}_F \sqrt{\tr(M^4)} \\
&\lesssim \max\{p,q\}^4 \tr(M) \norm{M}_F \sqrt{\tr(M^4)}.
\end{align*}
This implies that each term on the right-hand of~\eqref{eqn:Var_G term in U} grows geometrically.

\subparagraph*{Step 1b.} Next we deal with the second term in~\eqref{eqn:law of total variance}. We have
\begin{align*}
\E_G \tr\left((G^T MG)^2 \right) = \sum_{i,j} \E_G (G^T MG)_{ij}^2 &= \sum_{i,j} \E_G \bigg(\sum_{k,l} M_{kl} G_{ki} G_{lj}\bigg)^2 \\
&= \sum_{i,j} \sum_{k,l,k',l'} M_{kl} M_{k'l'} \E_G G_{ki} G_{lj} G_{k'i} G_{l'j}. 
\end{align*}
When $i\neq j$, for non-zero contribution, it must hold that $k=l$ and $k'=l'$ and thus the nonzero contribution is
\[
 \sum_{i\neq j} \sum_{k,l} M_{kl}^2 = q(q-1)\norm{M}_F^2.
\]
When $i = j$, the contribution is
\begin{equation}\label{eqn:i=j}
\sum_i \sum_{\substack{k,l,k',l'}} M_{kl} M_{k'l'} \E G_{ki} G_{li} G_{k'i} G_{l'i} = 2q\norm{M}_F^2 + q\tr^2(M).
\end{equation}
Hence
\[
\E_G \tr\left((G^T MG)^2 \right) = q(q+1)\norm{M}_F^2 + q\tr^2(M)
\]
and when $M$ is random,
\begin{equation}\label{eqn:second_term_in U}
\begin{aligned}
&\quad\, \Var\left(\left. \E \tr((G^T MG)^2)\right\vert M\right)\\
&= \Var\left(q(q+1)\norm{M}_F^2 + q\tr^2(M)\right)\\
&\leq q^2(q+1)^2\Var(\norm{M}_F^2) + q^2\Var(\tr^2(M)) + 2q^2(q+1)\sqrt{\Var(\norm{M}_F^2) \Var(\tr^2(M))}.
\end{aligned}
\end{equation}

\subparagraph*{Step 2a.} Note that the $\Var(\tr^2(M))$ term on the right-hand side of~\eqref{eqn:second_term_in U}. To bound this term, we examine the variance of $g(G)$, where $g(X) = \tr^2(X^T MX)$. We shall again calculate $\nabla g$. Note that
\[
\frac{\partial g}{\partial X_{rs}} = 2\tr(X^T M X) \sum_i\sum_{k, l} M_{kl} \frac{\partial}{\partial X_{rs}} X_{ki} X_{li}
\]
and
\[
\frac{\partial}{\partial X_{rs}}  (X_{k i} X_{l i}) = \begin{cases}
														X_{l i}, & (k,i)=(r,s)\text{ and }(l,i)\neq(r,s)\\
														X_{k i}, & (k,i)\neq (r,s)\text{ and }(l,i)=(r,s)\\ 
														2 X_{r s}, & (k,i)=(r,s)\text{ and }(l,i)=(r,s)\\
														0, & \text{otherwise}.
													   \end{cases}
\]
We have
\[
\frac{\partial g}{\partial X_{rs}} = 4\tr(X^T M X) \sum_{l} M_{rl} X_{ls} = 4\sum_{\substack{1\leq j\leq q\\ 1\leq l, u, v\leq p}} M_{uv} M_{rl} X_{ls} X_{uj} X_{vj}
\]
Next we calculate $\E (\partial g/\partial X_{rs})^2$ when $X$ is i.i.d.\ Gaussian. 
\[
\left(\frac{1}{4}\frac{\partial g}{\partial X_{rs}}\right)^2 = \sum_{\substack{j,l,u,v\\ j',l',u',v'}} M_{u v} M_{u' v'} M_{r l}  M_{r l'} \E  X_{ls} X_{l's} X_{uj} X_{vj} X_{u'j'} X_{v'j'}
\]
In order for the expectation in the summand to be non-zero, we must have one of the following cases: (1) $s\neq j\neq j'$, (2) $s = j \neq j'$, (3) $s = j' \neq j$, (4) $s \neq j = j'$, (5) $s = j = j'$.
We calculate the contribution in each case below.

Case 1: it must hold that $l = l'$, $u=v$ and $u'=v'$. The contribution is $(q-1)(q-2)B^{(1)}_{rs}$, where
\[
B^{(1)}_{rs} = \sum_{l,u,u'} M_{u u} M_{u' u'} M_{r l}^2 = \tr^2(M)\norm{M_{r,\cdot}}_2^2.
\]

Case 2: it must hold that $u'=v'$. The contribution is $(q-1)B^{(2)}_{rs}$, where
\begin{align*}
B^{(2)}_{rs} &= \sum_{l, l', u, u', v} M_{u v} M_{u' u'} M_{r l}  M_{r l'} \E X_{l s} X_{l' s} X_{u s} X_{v s} X_{u' j'}^2 \\
&= \tr(M) \bigg( \tr(M) \norm{M_{r,\cdot}}_2^2 + 2\sum_{l,l'} M_{ll'} M_{rl} M_{rl'}\bigg)
\end{align*}

Case 3: this gives the same bound as Case 2.

Case 4: it must hold that $l = l'$. The contribution is $(q-1)B^{(4)}_{rs}$, where
\[
B^{(4)}_{rs} = \sum_{l, u, u', v, v'} M_{u v} M_{u' v'} M_{r l}^2 \E X_{u j} X_{v j} X_{u' j} X_{v' j} 
    = 3\norm{M_{r,\cdot}}_2^2 \norm{M}_F^2
\]

Case 5: the contribution is $B^{(5)}_{rs}$, where
\[
B^{(5)}_{rs} = \sum_{\substack{l,u,v\\ l',u',v'}} M_{u v} M_{u' v'} M_{r l}  M_{r l'} \E  X_{ls} X_{us} X_{vs} X_{l's} X_{u's} X_{v's}.
\]
The only uncovered case is $l = u'$, $l' = v$, $u = v'$ and its symmetries. In such a case the contribution is at most
\[
C \sum_{l,u,v} M_{uv} M_{lu} M_{rl} M_{rv} = C\sum_u \langle M_{r,\cdot}, M_{u,\cdot}\rangle^2.
\]
Note that
\begin{align*}
\sum_{r,s} B_{rs}^{(1)} &= q\tr^2(M) \norm{M}_F^2, \\
\sum_{r,s} B_{rs}^{(2)} &= q\tr^2(M) \norm{M}_F^2 + 2q\tr(M)\sum_{l,l'} M_{ll'}\langle M_{l,\cdot},  M_{l',\cdot}\rangle\\
&\leq q\tr^2(M) \norm{M}_F^2 + 2q\tr(M) \norm{M}_F \sqrt{\tr(M^4)}, \\
\sum_{r,s} B_{rs}^{(4)} &= q\norm{M}_F^4, \\
\sum_{r,s} B_{rs}^{(5)} &\lesssim \sum_{r,s} B_{rs}^{(1)} + \sum_{r,s} B_{rs}^{(2)} + \tr(M^4).
\end{align*}
Therefore,
\begin{align*}
\frac{1}{16}\E \norm{\nabla g}_2^2 &\leq \sum_{r,s} ((q-1)(q-2)B_{rs}^{(1)} + (q-1)B_{rs}^{(2)} + (q-1)B_{rs}^{(4)} + B_{rs}^{(5)}) \\
&\lesssim q^3 \tr^2(M) \norm{M}_F^2 + q^2 \tr(M) \norm{M}_F \sqrt{\tr(M^4)} + q^2 \norm{M}_F^4 + q\tr(M^4).
\end{align*}
By Poincar\'e's inequality,
\begin{equation}\label{eqn:first_term in V}
\begin{aligned}
&\quad\ \Var_G(\tr^2(G^T MG)) \\
&\lesssim \E\norm{\nabla g}_2^2 \\
&\lesssim q^3 \tr^2(M) \norm{M}_F^2 + q^2 \tr(M) \norm{M}_F \sqrt{\tr(M^4)} + q^2 \norm{M}_F^4 + q\tr(M^4).
\end{aligned}
\end{equation}
Similar to before, each term on the right-hand side grows geometrically.

\subparagraph*{Step 2b.} Next we deal with $\Var_M(\E_G \tr^2\left(G^T MG \right) | M)$.
\[
\E \tr^2\left(G^T MG \right) = \E \bigg(\sum_{i, k, l} M_{kl} G_{ki} G_{li} \bigg)^2 
= \sum_{i,j} \sum_{k,l,k',l'} M_{kl} M_{k'l'} \E G_{ki} G_{li} G_{k'j} G_{l'j}. 
\]
When $i\neq j$, for non-zero contribution, it must hold that $k=l$ and $k'=l'$ and thus the nonzero contribution is
\[
 \sum_{i\neq j} \sum_{k,k'} M_{k k} M_{k'k'} = q(q-1)\tr^2(M).
\]
When $i = j$, the contribution is (this is exactly the same as \eqref{eqn:i=j} in Step 1b.)
\[
\sum_i \sum_{\substack{k, k',l,l'}} M_{kl} M_{k'l'} \E G_{ki} G_{li} G_{k'i} G_{l'i} 
= 2q\norm{M}_F^2 + q\tr^2(M).
\]
Hence
\[
\E \tr^2\left(G^T MG \right) = 2q\norm{M}_F^2 + q^2\tr^2(M)
\]
and when $M$ is random,
\begin{equation}\label{eqn:second_term in V}
\begin{aligned}
&\quad\ \Var\left(\left. \E \tr^2(G^T MG)\right\vert M\right) \\
&= \Var\left(2q\norm{M}_F^2 + q^2\tr^2(M)\right)\\
&\leq 4q^2\Var(\norm{M}_F^2) + q^4\Var(\tr^2(M)) + 2q^3\sqrt{\Var(\norm{M}_F^2) \Var(\tr^2(M))}.
\end{aligned}
\end{equation}

\subparagraph*{Step 3.} Let $U_r$ denote the variance of $\tr((A_r^T A_r)^2)$ and $V_r$ the variance of $\tr^2(A_r^T A_r)$. Combining \eqref{eqn:law of total variance}, \eqref{eqn:Var_G term in U}, \eqref{eqn:second_term_in U}, \eqref{eqn:first_term in V}, \eqref{eqn:second_term in V}, we have the following recurrence relations, where $C_1,C_2,C_3,C_4 > 0$ are absolute constants.
\begin{align*}
U_{r+1} &\leq C_1 P_{r} + 2 U_r + \frac{1}{d_r^2} V_r + \frac{3}{d_r}\sqrt{U_r V_r} \\
V_{r+1} &\leq C_2 Q_{r} + \frac{1}{d_r^2} U_r + V_r + \frac{2}{d_r}\sqrt{U_r V_r} \\
P_{r+1} &\leq C_3 P_r\\
Q_{r+1} &\leq C_4 Q_r\\
U_0 &= V_0 = 0
\end{align*}
In the base case, set $M = I_p$ (the $p\times p$ identity matrix in~\eqref{eqn:Var_G term in U}) and note that the second term in~\eqref{eqn:law of total variance} vanishes. We see that $P_1 \lesssim (p^3 q + p q^3)/d_1^4$  after proper normalization. (Alternatively we can calculate this precisely, see Appendix~\ref{sec:variance r = 2}.) Similarly we have $Q_1 \lesssim p^3q^3/d_1^4$. Note that $Q_1/d_1^2 \lesssim (p^3 q + p q^3)/d_1^4$. Now, we can solve that 
\[
U_{r+1} \leq C^r \frac{p^3 q + p q^3}{d_1^4}
\]
for some absolute constant $C>0$.

\bibliographystyle{plainurl}
\bibliography{reference}

\appendix
\section{Omitted Calculations in Section~\ref{sec:mean}}\label{sec:omitted_mean}
\begin{gather*}
S_1(p,q) = 3\sum_i \E B_{1i}^4 + 3\sum_{i\neq j} \E B_{1i}^2 B_{1j}^2 = 3d T_1(p,d) + 3d(d-1)T_4(p,d)\\
\begin{aligned}
S_3(p,q) = \E A_{11}^2 A_{21}^2 &= \E \bigg(\sum_i B_{1i}G_{i1}\bigg)^2 \bigg(\sum_k B_{2k} G_{k1}\bigg)^2 \\
&= \sum_{i,j,k,l} \E B_{1i}B_{1j}B_{2k}B_{2l} \E G_{i1}G_{j1}G_{k1}G_{l1} \\
&= 3\sum_i \E B_{1i}^2 B_{2i}^2 + \sum_{i\neq j} \E B_{1i}^2 B_{2j}^2 + 2\sum_{i\neq j} \E B_{1i}B_{2i}B_{1j}B_{2j} \\
&= 3dT_3(p,d) + d(d-1)T_5(p,d) + 2d(d-1)T_6(p,d)
\end{aligned}\\
\begin{aligned}
S_4(p,q) = \E A_{11}^2 A_{12}^2 &= \E \bigg(\sum_i B_{1i}G_{i1}\bigg)^2 \bigg(\sum_k B_{1k} G_{k2}\bigg)^2 \\
&= \sum_{i,j,k,l} \E B_{1i}B_{1j}B_{1k}B_{1l} \E G_{i1}G_{j1}G_{k2}G_{l2} \\
&= \sum_i \E B_{1i}^4 + \sum_{i\neq j} \E B_{1i}^2 B_{1j}^2 
= d T_1(p,d) + d(d-1)T_4(p,d).
\end{aligned}\\
\begin{aligned}
S_5(p,q) = \E A_{11}^2 A_{22}^2 &= \E \bigg(\sum_i B_{1i}G_{i1}\bigg)^2 \bigg(\sum_k B_{2k} G_{k2}\bigg)^2 \\
&= \sum_{i,j,k,l} \E B_{1i}B_{1j}B_{2k}B_{2l} \E G_{i1}G_{j1}G_{k2}G_{l2} \\
&= \sum_{i,j} B_{1i}^2 B_{2j}^2
= d T_3(p,d) + d(d-1) T_5(p,d)
\end{aligned}\\
\begin{aligned}
S_6(p,q) = 
\E A_{11}A_{12}A_{21}A_{22} &= \sum_{i,j,k,l} \E B_{1i}B_{1j}B_{2k}B_{2l} \E G_{i1}G_{j2}G_{k1}G_{l2} \\
&= \sum_i \E B_{1i}^2 B_{2i}^2 + \sum_{i\neq j} \E B_{1i} B_{1j} B_{2i} B_{2j} \\
&= d T_3(p,d) + d(d-1)T_6(p,d)
\end{aligned}
\end{gather*}

\section{Omitted Calculations in Section~\ref{sec:variance}}\label{sec:omitted_variance}

\subparagraph{In Step 1a.} 
\begin{gather*}
\begin{aligned}
B^{(2)}_{r,s} &= \sum_{l,l'} \sum_{u,v,v'} M_{uv}M_{uv'}M_{rl} M_{rl'}\E X_{us}^2 X_{vj} X_{lj} X_{v'j} X_{l'j} \\
&= \underbrace{\sum_{l\neq l'}\sum_u M_{ul}M_{ul'}M_{rl} M_{rl'}}_{v=l \neq v'=l'} + \underbrace{\sum_{l}\sum_{\substack{u\\ v\neq l}} M_{uv}^2 M_{rl}^2}_{v=v'\neq l=l'} + \underbrace{\sum_{l\neq l'}\sum_u M_{ul'} M_{ul} M_{rl} M_{rl'}}_{v=l'\neq l=v'}\\
&\qquad\qquad\qquad+ 3\underbrace{\sum_{l, u} M_{ul}^2 M_{rl}^2}_{v=v'=l=l'} \\
&= \left(\sum_{u,v} M_{uv}^2\right)\left(\sum_l M_{rl}^2\right) + 2\sum_{l,l',u} M_{ul'} M_{ul} M_{rl} M_{rl'} \\
&= \norm{M}_F^2 \norm{M_{r,\cdot}}_2^2 + 2\sum_u \langle M_{u,\cdot},M_{r,\cdot}\rangle^2.
\end{aligned}\\
\begin{aligned}
B^{(3)}_{r,s} &= \sum_{j'\neq s}\left[\sum_{l, l'} \sum_{u,v} M_{u v} M_{u'l'} M_{rl} M_{rl'}\E X_{us} X_{u's} X_{vs} X_{ls} X_{l'j'}^2\right]\\
&= \underbrace{\sum_{l, l'}\sum_{u\neq l} M_{ul} M_{u l'} M_{rl} M_{r l'}}_{u=u'\neq v=l} +
 \underbrace{\sum_{l, l'}\sum_{u\neq l} M_{uu} M_{l l'} M_{rl} M_{rl'}}_{u=v\neq u'=l} +
  \underbrace{\sum_{l, l'} \sum_{v} M_{l v} M_{vl'} M_{rl} M_{rl'}}_{u=l\neq u'=v} \\
&\qquad +
  3\underbrace{\sum_{l, l'} M_{l l} M_{l l'} M_{rl} M_{rl'}}_{u=u'=v=l} \\
&= \sum_{l, l'}\sum_{u} (2M_{ul} M_{ul'} + M_{uu} M_{l l'}) M_{rl} M_{rl'} \\
&= \sum_{l, l'} (2\langle M_{l,\cdot}, M_{l',\cdot}\rangle + \tr(M) M_{l l'}) M_{rl} M_{rl'}.
\end{aligned}\\
\begin{aligned}
\sum_{r,s} B^{(3)}_{r,s} &= \sum_{r,s} \sum_{l, l'}(2\langle M_{l,\cdot}, M_{l',\cdot}\rangle + \tr(M) M_{l'l'}) M_{rl} M_{rl'}\\
&= q \sum_{l, l'} (2\langle M_{l,\cdot}, M_{l',\cdot}\rangle + \tr(M) M_{l l'}) \sum_{r} M_{rl} M_{rl'} \\
&= q\sum_{l, l'} (2\langle M_{l,\cdot}, M_{l',\cdot}\rangle + \tr(M) M_{l l'}) \langle M_{l,\cdot},M_{l',\cdot}\rangle \\
&= q\sum_{l, l'} 2\langle M_{l,\cdot}, M_{l',\cdot}\rangle^2 + q\tr(M)  \sum_{l, l'} M_{l l'}\langle M_{l,\cdot},M_{l',\cdot}\rangle \\
&\leq 2q\tr(M^4) + q \tr(M) \bigg(\sum_{l,l'} M_{l l'}^2\bigg)^{\frac{1}{2}} \bigg(\sum_{l,l'} \langle M_{l,\cdot}, M_{l',\cdot}\rangle^2 \bigg)^{\frac{1}{2}} \\
&\leq 2q\tr(M^4) + q \tr(M) \norm{M}_F \sqrt{\tr(M^4)}
\end{aligned}
\end{gather*}

\subparagraph{In Step 1b.}
\begin{align*}
&\quad\ \sum_i \sum_{\substack{k,l\\ k',l'}} M_{kl} M_{k'l'} \E G_{ki} G_{li} G_{k'i} G_{l'i} \\
&= \sum_i \bigg( \underbrace{\sum_{k\neq l} M_{kl}^2}_{k=k'\neq l=l'} + \underbrace{\sum_{k\neq l} M_{kl}^2}_{k=l'\neq k'=l} + \underbrace{\sum_{k\neq l} M_{kk}M_{k'k'}}_{k=l\neq k'=l'} + 3\underbrace{\sum_{k} M_{kk}^2}_{k=k'=l=l'}\bigg)  \\
&= \sum_i \bigg( \sum_{k, l} M_{kl}^2 + \sum_{k, l} M_{kl}^2 + \sum_{k, l} M_{kk}M_{k'k'} \bigg)  \\
&= \sum_i (2\norm{M}_F^2 + \tr^2(M))\\
&= 2q\norm{M}_F^2 + q\tr^2(M).
\end{align*}

\subparagraph{In Step 2a.}
\begin{align*}
B^{(2)}_{rs} &= \sum_{l, l', u, u', v} M_{u v} M_{u' u'} M_{r l}  M_{r l'} \E X_{l s} X_{l' s} X_{u s} X_{v s} X_{u' j'}^2 \\
&= \underbrace{\sum_{\substack{l\neq u\\ u'}} M_{u u} M_{u' u'} M_{r l}^2}_{l=l'\neq u=v} + \underbrace{\sum_{\substack{l\neq l'\\ u'}} M_{l l'} M_{u' u'} M_{r l}  M_{r l'}}_{l=u\neq l'=v}  + \underbrace{\sum_{\substack{l\neq l'\\ u'}} M_{l l'} M_{u' u'} M_{r l}  M_{r l'}}_{l=v\neq l'=u} \\
&\qquad\qquad\qquad + 3\underbrace{\sum_{l,u'} M_{l l} M_{u' u'} M_{r l}^2}_{l=u= l'=v}\\
&= \tr(M) \left(\sum_{l,u} M_{uu} M_{rl}^2 + 2\sum_{l,l'} M_{ll'} M_{rl} M_{rl'}\right)\\
&= \tr(M) \left( \tr(M) \norm{M_{r,\cdot}}_2^2 + 2\sum_{l,l'} M_{ll'} M_{rl} M_{rl'}\right)
\end{align*}

\begin{align*}
B^{(4)}_{rs} &= \sum_{l, u, u', v, v'} M_{u v} M_{u' v'} M_{r l}^2 \E X_{u j} X_{v j} X_{u' j} X_{v' j} \\
    &= \left(\sum_l M_{r l}^2\right)\left(\underbrace{\sum_{u, v} M_{u v}^2}_{u=u'\neq v=v'} + \underbrace{\sum_{u, v} M_{u v}^2}_{u=v'\neq u=v} + \underbrace{\sum_{u, v} M_{u v}^2}_{u=v\neq u'=v'} + \underbrace{3\sum_{u} M_{u u}^2}_{u=v = u'=v'}\right) \\
    &= 3\left(\sum_l M_{r l}^2\right)\sum_{u,v} M_{u,v}^2 \\
    &= 3\norm{M_{r,\cdot}}_2^2 \norm{M}_F^2
\end{align*}


\section{Exact Variance when $r=2$}\label{sec:variance r = 2}

Suppose that $A$ is rotationally invariant under both left- and right-multiplication of an orthogonal matrix. Define
\begin{align*}
U_1(p, q) &= \Var( (A^TA)_{ii}^2 )\\
U_2(p, q) &= \Var( (A^TA)_{ij}^2 )\quad i\neq j \\
U_3(p, q) &= \cov( (A^TA)_{ii}^2, (A^TA)_{ik}^2 ) \quad i\neq k \quad \text{(same row, one entry on diagonal)} \\
U_4(p, q) &= \cov(  (A^TA)_{ij}^2, (A^TA)_{ik}^2 ) \quad j\neq k \quad \text{(same row, both entries off-diagonal)} \\
U_5(p, q) &= \cov(  (A^TA)_{ii}^2, (A^TA)_{jj}^2 ) \quad i\neq j \quad \text{(diff.\ rows and cols, both entries on diagonal)} \\
U_6(p, q) &= \cov(  (A^TA)_{ii}^2, (A^TA)_{jk}^2 ) \quad i\neq j\neq k \quad \text{(diff.\ rows and cols, one entry on diagonal)} \\
U_7(p, q) &= \cov(  (A^TA)_{ij}^2, (A^TA)_{kl}^2 ) \quad i\neq j\neq k\neq l \quad \text{(diff.\ rows and cols, nonsymmetric around  diag.)}
\end{align*}
It is clear that they are well-defined.
\begin{align*}
&\quad\ \Var(\tr((A^T A)^2)) \\
&= \Var\bigg( \sum_{i,j} (A^TA)_{ij}^2 \bigg)\\
&= \sum_{i,j,k,l} \cov( (A^TA)_{ij}^2, (A^TA)_{kl}^2 ) \\
&= \sum_{i,j} \Var( (A^TA)_{ij}^2 ) + 2\sum_i\sum_{j\neq l} \cov( (A^TA)_{ij}^2, (A^TA)_{il}^2 )  + \sum_{\substack{i\neq k\\ j\neq l}}\cov(\E (A^TA)_{ij}^2, (A^TA)_{kl}^2)\\
&= q\Var((A^TA)_{11}^2) + q(q-1)\Var(\E(A^TA)_{12}^2) \\
&\qquad + 2\left[ 2q(q-1)\cov((A^TA)_{11}^2, (A^TA)_{12}^2) + q(q-1)(q-2)\cov((A^TA)_{12}^2, (A^TA)_{13}^2) \right] \\
&\qquad + q(q-1)\cov(\E(A^TA)_{11}^2,(A^TA)_{22}^2) + q(q-1)\cov(\E (A^T A)_{12}^2, (A^TA)_{21}^2) \\
&\qquad\quad + 2q(q - 1)(q - 2)\cov((A^TA)_{11}^2, (A^TA)_{23}^2) \\ 
&\qquad\quad + 2q(q - 1)(q - 2)\cov((A^TA)_{12}^2, (A^TA)_{31}^2) \\
&\qquad\quad + q(q-1)(q-2)(q-3)\E(A^TA)_{12}^2(A^TA)_{34}^2\\
&= q U_1(p,q) + q(q-1) U_2(p,q) + 2q(q-1)(2U_3(p,q) + (q-2)U_4(p,q)) \\
&\qquad + q(q-1)(U_5(p,q) + U_2(p,q)) + 2q(q-1)(q-2)(U_6(p,q) + U_4(p,q)) \\
&\qquad + q(q-1)(q-2)(q-3)U_7(p,q) \\
&= q U_1(p,q) + q(q-1) (2U_2(p,q) + 4U_3(p,q) + U_5(p,q))  \\
&\qquad + 2q(q-1)(q-2)(2U_4(p,q) + U_6(p,q)) + q(q-1)(q-2)(q-3)U_7(p,q).
\end{align*}
Let us calculate $U_1,\dots,U_7$ for a $p\times q$ Gaussian random matrix $G$.
\begin{gather*}
\begin{aligned}
U_1(p,q) = \E(G^T G)_{11}^4 - (\E(G^TG)_{11}^2)^2 &= \E \norm{G_1}_2^8 - (\E\norm{G_1}_2^4)^2 \\
&= p(p+2)(p+4)(p+6) - (p(p+2))^2\\
&= 8p(p+2)(p+3)
\end{aligned}\\
\begin{aligned}
U_2(p,q) = \E(G^T G)_{12}^4 - (\E(G^TG)_{12}^2)^2 
&= \E \left(\sum_r G_{r1}G_{r2}\right)^4 - (\E\langle G_1,G_2\rangle^2)^2\\
&= \sum_{r,s,t,u} \E G_{r1} G_{s1} G_{t1} G_{u1} G_{r2} G_{s2} G_{t2} G_{u2} - p^2\\
&= 3\sum_{r\neq t} \E G_{r1}^2 G_{t1}^2 G_{r2}^2 G_{t2}^2  + \sum_r G_{r1}^4 G_{r2}^4 - p^2\\
&= 3p(p-1) + 9p - p^2 = 2p(p+3).
\end{aligned}\\
\begin{aligned}
U_3(p,q) &= \E(G^T G)_{11}^2 (G^T G)_{12}^2 - \E(G^T G)_{11}^2 \E(G^T G)_{12}^2\\
 &= \E (G_1^T G_1)^2 G_1^T G_2 G_2^T G_1 - \E\norm{G_1}_2^4 \E \langle G_1,G_2\rangle^2 \\
&= \E (G_1^T G_1)^2 G_1^T (\E G_2 G_2^T) G_1 - p(p+2)\cdot p\\
&= \E (G_1^T G_1)^3 - p^2(p+2)\\
&= \E \norm{G_1}_2^6 - p^2(p+2) = p(p+2)(p+4) - p^2(p+2) = 4p(p+2)
\end{aligned}\\
\begin{aligned}
U_4(p,q) &= \E(G^T G)_{12}^2 (G^T G)_{13}^2 - \E(G^T G)_{12}^2 \E(G^T G)_{13}^2\\
&= \E G_1^T G_2 G_2^T G_1 G_1^T G_3 G_3^T G_1 - p^2\\
&= \E G_1^T \E(G_2 G_2^T) G_1 G_1^T \E(G_3 G_3^T) G_1 -p^2 \\
&= \E (G_1^T G_1)^2 - p^2 = \E \norm{G_1}_2^4 - p^2 = p(p+2) - p^2 = 2p
\end{aligned}\\
\begin{aligned}
U_5(p,q) &= \E(G^T G)_{11}^2 (G^T G)_{22}^2 - \E(G^T G)_{11}^2 \E(G^T G)_{22}^2 \\
&= \E \norm{G_1}_2^4 \norm{G_2}_2^4 - \E \norm{G_1}_2^4 \norm{G_2}_2^4 = 0
\end{aligned}\\
\begin{aligned}
U_6(p,q) &= \E(G^T G)_{11}^2 (G^T G)_{23}^2 - \E(G^T G)_{11}^2 \E (G^T G)_{23}^2 \\
&= \E \norm{G_1}_2^4 \langle G_2,G_3\rangle^2 -\E \norm{G_1}_2^4 \E \langle G_2,G_3\rangle^2 = 0
\end{aligned}\\
\begin{aligned}
U_7(p,q) &= \E(G^T G)_{12}^2 (G^T G)_{34}^2 - \E(G^T G)_{12}^2 \E (G^T G)_{34}^2 \\
&= \E \langle G_1,G_2\rangle^2 \langle G_3,G_4\rangle^2 - \E \langle G_1,G_2\rangle^2 \E \langle G_3,G_4\rangle^2 = 0
\end{aligned}
\end{gather*}
Therefore
\begin{align*}
\Var(\tr((G^T G)^2)) &= q U_1 + q (q - 1) (2 U_2 + 4 U_3 + U_5) + 
 2 q (q - 1) (q - 2) (2 U_4 + U_6) \\
 &\qquad + q (q - 1) (q - 2) (q - 3) U_7 \\
 &= q U_1 + q (q - 1) (2 U_2 + 4 U_3) +  4 q (q - 1) (q - 2) U_4 \\
 &= 4 p q (5 + 5 p + 5 q + 2 p^2 + 5 p q + 2 q^2).
\end{align*}
When $r=2$, recalling that $\E(A_2 - A_1) = (1+o(1)) p^2q^2/d^3$ (see~\eqref{eqn:mean_estimate}), we have that
\[
\frac{\sqrt{\Var(\tr((\frac{1}{\sqrt d}G^T \cdot \frac{1}{\sqrt d}G)^2))}}{p^2q^2/d^3} \leq \frac{6d}{\max\{p,q\}^{\frac{1}{2}}\min\{p,q\}^{\frac{3}{2}}}.
\]
If the right-hand side above is at most a small constant $c$, we can distinguish $A_2$ from $A_1$ with probability at least a constant.

\end{document}